\documentclass[a4paper]{article}

\usepackage[english]{babel}
\usepackage[utf8]{inputenc}
\usepackage{amsmath}
\usepackage{amsthm}
\usepackage{mathrsfs}
\usepackage{mathtools}
\usepackage{amssymb}
\usepackage{theoremref}

\usepackage{graphicx}
\numberwithin{equation}{section}

\usepackage{pdfsync}

\usepackage[colorinlistoftodos]{todonotes}
\usepackage{xcolor}

\newtheorem{theorem}{Theorem}[section]
\newtheorem{prop}[theorem]{Proposition}
\newtheorem{defn}[theorem]{Definition}
\newtheorem{cor}[theorem]{Corollary}
\newtheorem{lem}[theorem]{Lemma}

\theoremstyle{remark}
\newtheorem*{oss}{{\bf Remark}}

\newcommand{\taubar}{{\bar \tau}}

\def\deltauno{\delta}
\def\deltadue{\delta}

\begin{document}
\title{On the existence of $L^2$-valued thermodynamic entropy solutions for a
  hyperbolic system with boundary conditions}
\author{
  Stefano Marchesani\\
Stefano Olla}
\maketitle

\abstract
We prove existence of $L^2$-weak solutions of a 
quasilinear wave equation with boundary conditions.
This describes the isothermal evolution of a one dimensional non-linear elastic material,
attached to a fixed point on one side and subject to a force (tension) applied to the other side.
The $L^2$-valued solutions appear naturally when studying the hydrodynamic limit
from a microscopic dynamics of a chain of anharmonic springs connected to a thermal bath.
The proof of the existence is done using a vanishing viscosity approximation
with extra Neumann boundary conditions added.
In this setting 
we obtain 
 a uniform a priori estimate in $L^2$,
allowing us to use $L^2$ Young measures,
together with the classical tools of compensated compactness.
We then prove that the viscous solutions converge to weak solutions of
the quasilinear wave equation
strongly in $L^p$, for any $p \in [1,2)$,
that satisfy, in a weak sense, the boundary conditions. 
Furthermore, these solutions satisfy, beside the local Lax entropy condition,
the Clausius inequality: the change of the free energy is bounded
by the work done by the boundary tension.
In this sense they are the correct thermodynamic solutions,
and we conjecture their uniqueness.
\bigbreak

Keywords: hyperbolic conservation laws, quasi-linear wave equation,
boundary conditions, weak solutions,
vanishing viscosity, compensated compactness, entropy solutions, Clausius inequality.
\\
Mathematics Subject Classification numbers: 35L40, 35D40

{\let\thefootnote\relax
\footnote{{\today}}
}

\section{Introduction}
\label{sec:intro}

The problem of existence of weak solutions for hyperbolic systems of conservation law
in a bounded domain has been studied for solutions that are of bounded variation or in $L^\infty$
\cite{chen-frid1999}.
In the scalar case some works extend to $L^\infty$ solutions,
obtained from viscous approximations \cite{otto}.
But viscous approximations require extra boundary conditions, that are usually taken of Dirichlet type.

We present here an approach based on viscosity approximations, where the extra boundary conditions
are of Neumann type, to reflect the conservative nature of the viscous approximation.
We consider here the quasilinear wave equation
\begin{equation} \label{eq:hyperintro}
\begin{cases}
r_t -p_x = 0
\\
p_t-\tau(r)_x = 0
\end{cases}, \qquad (t,x) \in \mathbb R_+ \times [0,1]
\end{equation}
where $\tau(r)$ is a strictly increasing regular function of $r$ such that $0 < c_1 \le \tau'(r)  \le c_2$,
for some constant $c_1,c_2$. In section \ref{sec:definition} we will require
some more technical assumption for $\tau$.
We add to the system the following boundary conditions:
\begin{equation}
p(t,0)= 0, \qquad \tau(r(t,1))= \taubar(t)
\end{equation}
and initial data
\begin{equation}
r(0,x)=r_0(x), \qquad p(0,x)= p_0(x).
\end{equation}
The boundary tension $\taubar : \mathbb R_+ \to \mathbb R$
is smooth and bounded with bounded derivative.

The equations \eqref{eq:hyperintro} describe the isothermal evolution of an elastic material
in Lagrangian coordinates. The \emph{material} point $x\in [0,1]$ has a \emph{volume strain}
$r(t,x)$ at time $t$ (that can also have negative values), and momentum (velocity) $p(t,x)$. The 
Eulerian position of the material point $x$, with respect to the position of the particle $0$,
is given by $q(t,x) = \int_0^x r(t,y) dy$,
so that we can identify the position of the material point $x=1$ as the total extension of the material:
\begin{equation}
  \label{eq:6}
  L(t) = q(t,1) = \int_0^1 r(t,y) dy.
\end{equation}


Let $T<\infty$ be given and arbitrary, and define $Q_T := [0,T] \times [0,1]$.
We shall construct weak solutions $\bar u(t,y) = \left(\bar r(t,y), \bar p(t,y)\right), (t, y)\in Q_T$,
to the quasilinear wave equation
such that $\bar u(t,\cdot) \in L^2(0,1)$ for all $t\le T$
and satisfy the initial and boundary conditions in the following weak sense:
\begin{equation} \label{eq:maineqr}
  \int_0^1\varphi(t,x)\bar r(t,x)dx - \int_0^1\varphi(0,x)r_0(x)dx =
  \int_0^t \int_0^1 \left(  \varphi_s \bar r-  \varphi_x\bar p \right) dx ds
\end{equation}
\begin{equation} \label{eq:maineqp}
  \int_0^1\psi(t,x)\bar p(t,x)dx - \int_0^1\psi(0,x)p_0(x) dx =
  \int_0^t \int_0^1    \left( \psi_s \bar p -  \psi_x \tau(\bar r) \right) dx ds
  + \int_0^t  \psi(s,1) \taubar(s) ds
\end{equation}
for all  functions $\varphi, \psi \in C^1(Q_T)$ 
such that  $\varphi(t,1)  = \psi(t,0) = 0$ for all $ t \ge 0$.

Define the \emph{free energy} of the system, associated to a profile $u(x) = (r(x), p(x))\in  L^2(0,1)$, as
\begin{equation}
\label{eq:freeenergy}
\mathcal F( u) := \int_0^1  \left( \frac{ p^2(x)}{2} + F( r(x))  \right) dx
\end{equation}
where $F(r)$ is a primitive of $\tau(r)$ ($F'(r) = \tau(r)$), such that
$ \dfrac{c_1}{2}r^2 \le F(r) \le \dfrac{c_2}{2} r^2$
for any $ r \in \mathbb R$. This is possible thanks to the bounds we required on $\tau'$.

The solution $\bar u$ of \eqref{eq:maineqp} that we obtain has the following properties:
\begin{itemize}
\item
$ \bar u\in  L^\infty(0,T; L^2(0,1))$ 
 \item $\bar u(0,x) =u_0(x)$ for a.e. $x$;
 \item 
   For any $\phi \in C^1([0,1])$, the application
\begin{equation}
t \mapsto \int_0^1 \phi(x) \bar u(t,x)dx
\end{equation}
is Lipschitz continuous over  $[0,T]$;
\item $\bar u$ satisfies \emph{Clausius inequality}:
\begin{equation} \label{eq:clausius}
\mathcal F(\bar u(t))- \mathcal F(u_0) \le  W(t), \qquad  
\forall t \in[ 0,T]
\end{equation}

where $u_0 = (r_0,p_0)$ and
\begin{equation} \label{eq:lavoro}
W(t) :=  -\int_0^t  \taubar'(s) \int_0^1 \bar r(s,x) dxds + \taubar(t) \int_0^1 \bar r(t,x)dx- \taubar(0) \int_0^1 r_0(x)dx
\end{equation}
is the work done by the external tension up to time $t$. In this sense we call our solution a
\emph{thermodynamic entropy solution}. For general discussion of the connection of such thermodynamic solutions
to the usual definition of entropic solutions, see \cite{evans2004survey} and \cite{ball2013entropy}.
\end{itemize}

\begin{oss}

If $\bar r(t,x)$ is differentiable with respect to time, we may perform an integration by parts and obtain
\begin{equation}
W(t) = \int_0^t \taubar(s) d L(s).
\end{equation}
This recovers the usual mechanical definition of the work.
\end{oss}

The construction of the solution is obtained from the following viscosity approximation
\begin{equation} \label{eq:paraintro}
\begin{cases}
r^\delta_t -p^\delta_x = \deltauno r^\delta_{xx}
\\
p^\delta_t-\tau(r^\delta)_x = \deltadue p^\delta_{xx}
\end{cases}, \qquad (t,x) \in \mathbb R_+ \times [0,1]
\end{equation}
with boundary conditions
\begin{equation}
p^\delta(t,0)=0, \quad \tau(r^\delta(t,1))=\taubar(t), \quad p_x^\delta(t,1) = 0, \quad r_x^\delta(t,0)=0
\end{equation}
and initial data
\begin{equation}
r^\delta(0,x) = r^\delta_0(x), \qquad p^\delta(0,x) = p^\delta_0(x)
\end{equation}
such that $r_0^\delta$ and $p_0^\delta$ are compatible with the boundary conditions, regular enough
(see \eqref{eq:vic1} and \eqref{eq:vic2}) and converge to $r_0$ and $p_0$, respectively, as $\delta \to 0$.

Note that in the viscous approximation we have added two Neumann boundary conditions,
that reflect the \emph{conservative} 
nature of the viscous perturbation. Under these conditions we have 
\begin{equation}
\int_0^1 |u^\delta(t,x)|^2dx  +
\delta \int_0^t \int_0^1| u_x^\delta(s,x)|^2 dx ds 
\le C,\qquad \forall t \ge 0
\end{equation}
where $C$ is independent of $t$ and $\deltauno$.
It is thus clear that $\{u^\delta\}_{\delta >0}$ and $\{\sqrt \delta u_x^\delta\}_{\delta >0}$
are uniformly bounded in $L^2(Q_T)$.
Then we rely on the existence of a family of bounded 
Lax entropy-entropy fluxes  as in \cite{Shearer1}, \cite{serre86} and \cite{SerreShearer},
that allows us to apply the compensated compactness in the 
$L^2$ version. The conditions assumed on $\tau(r)$ are in fact those required to apply \cite{Shearer1} results.
Under a slight different set of conditions, another $L^p$ extension of the compensated compactness
argument can be found in \cite{lin92}.

\subsection{Physical motivations}
\label{sec:physical-motivations}

The problem arises naturally considering hydrodynamic
limit for a non-linear chain of anharmonic oscillators
in contact with a heat bath at a given temperature \cite{MO1,MO2}.
This microscopic dynamics models an isothermal transformation with two locally
conserved quantities that evolve, on the macroscopic scale, following \eqref{eq:hyperintro}. 

Consider $N+1$ particles on the real line and, for $i=0, \dots N$, call $q_i$ and $p_i$
the positions and the momenta of the $i$-th particle, respectively. 
Particles $i$ and $i-1$ interact via a nonlinear potential $V(q_i-q_{i-1})$.
Particle $i=0$ is at position $q_0 = 0$ and does not move, i.e. $p_0(t) = 0$.
There is a time dependent force (tension) $\taubar(t)$ acting on the last particle.
Then, defining $r_i := q_i-q_{i-1}$ we have a system with Hamiltonian
\begin{equation}\label{eq:ham}
  H_N (t) = \sum_{i=1}^N \left( \frac{p_i^2}{2}+V(r_i) \right) - \taubar(t) \sum_{i=1}^N r_i.
\end{equation}


 The interaction with a heat bath at temperature $\beta^{-1}$ is modeled by a stochastic perturbation
 of the dynamics, that acts as a \emph{microscopic stochastic viscosity}.
 Defining the discrete gradient and laplacian as
 \begin{equation*}
	\nabla a_i = a_{i+1} - a_i, \qquad \Delta a_i = a_{i+1} + a_{i-1} - 2 a_i,
\end{equation*}
the evolution equations are then given by the following system of stochastic differential equations:
\begin{equation}\label{eq:sde}
\begin{cases}
d r_1 =  p_1 d t + \deltauno  \nabla V'(r_1) d t - \sqrt{2 \beta^{-1}\deltauno} \, d \widetilde w_1
\\
d r_i =  \nabla p_{i-1} d t + \deltauno \Delta V'(r_i) d t - \sqrt{2 \beta^{-1} \deltauno}\,
\nabla d\widetilde w_{i-1}, & 2 \le i \le N-1\\
d r_N =  \nabla p_{N-1} dt + \deltauno \left(\taubar(t) + V'(r_{N-1}) - 2V'(r_N)\right)
- \sqrt{2 \beta^{-1} \deltauno}\, \nabla d\widetilde w_{N-1},
\\
d p_1 =  \nabla V'(r_1) d t + \deltauno \left(p_2 - 2p_1\right) dt - \sqrt{2 \beta^{-1} \deltadue}
\,  \nabla dw_{1},\\
d p_j =  \nabla V'(r_j) d t +\deltauno \Delta p_j  d t - \sqrt{2 \beta^{-1} \deltadue}
\,  \nabla dw_{j-1}, & 2 \le j \le N-1
\\
d p_N = (\taubar(t)-V'(r_N)) d t -\deltadue \nabla p_{N-1}d t
+ \sqrt{2\beta^{-1}\deltadue }\, d  w_{N-1}
\end{cases}
\end{equation}
Here $\beta^{-1} >0$ is the temperature of the heat bath,
and $\{w_i\}_{i=1}^{N-1}$, $\{\widetilde w_i\}_{i=1}^{N-1}$ are
families of independent Brownian motions.
The parameter $\deltadue$ is the intensity of the action of the heat bath, and is
chosen depending on $N$ such that  $\deltadue \sim o(N)$. When $\deltadue = 0$,
equations \eqref{eq:sde} are just the Newton deterministic equations for the Hamiltonian  \eqref{eq:ham}.
Notice the correspondence of the boundary conditions in \eqref{eq:sde} with the one chosen in
\eqref{eq:paraintro}.

One of the effects of the action of the stochastic heat bath is to fix, in a large time scale, the variance of the velocities (i.e. the temperature) at $\beta^{-1}$, and establish a local equilibrium, where
space-time averages of $V'(r_i)$ around a macroscopic particle number $[Nx]$ at a macroscopic
time $Nt$ converges to the equilibrium tension $\tau(r(t,x),\beta)$ at temperature
$\beta^{-1}$ and volume  stretch $r(t,x)$. Since $\beta$ is fixed by the heat bath
and do not evolve in time, we drop it from the notation in the sequel.

The hydrodynamic limit consists in proving that, for any continuous function $G(x)$ on $[0,1]$, 
\begin{equation}
  \label{eq:4}
  \frac 1N \sum_{i=1}^N G\left(\frac iN\right) \begin{pmatrix} r_i(Nt) \\ p_i(Nt) \end{pmatrix}
  \ \mathop{\longrightarrow}_{N\to\infty} \ \int_0^1 G(x)  \begin{pmatrix} r(t,x) \\ p(t,x) \end{pmatrix}\; dx, 
\end{equation}
in probability, with $(r(t,x), p(t,x))$  satisfying \eqref{eq:maineqr}, \eqref{eq:maineqp}.
Of course a complete proof would require the uniqueness of such $L^2$ valued solutions that satisfy \eqref{eq:clausius}: this remains an open problem. 
The results contained in \cite{MO2} states that the limit distribution of the empirical distribution 
defined on the RHS of \eqref{eq:4}, concentrates on the possible solutions of \eqref{eq:maineqr}
and \eqref{eq:maineqp} that satisfy \eqref{eq:clausius}.
Since we have no uniqueness result, we cannot assure that the solutions constructed in the present paper coincide
with those obtained with the hydrodynamic limit from  \eqref{eq:sde}. One can however conjecture that this is the case.



This stochastic model was already considered by Fritz \cite{Fritz1} 
in the infinite volume without boundary conditions, and in \cite{MO1},
 but without the characterisation of the boundary conditions. 

 In the hydrodynamic limit only $L^2$ bounds are available and we are constrained to consider
 $L^2$ valued solutions. Since these solutions do not have definite values on the boundary,
 boundary conditions have only a dynamical meaning in the sense of an evolution in $L^2$ given by
 \eqref{eq:maineqr}, \eqref{eq:maineqp}.

\section{Hyperbolic system and the existence of weak solutions}
\label{sec:definition}

For $r,p : \mathbb R_+ \times [0,1] \to \mathbb R$, consider the hyperbolic system
\begin{equation} \label{eq:psystem}
\begin{cases}
r_t - p_x=0
 \\
p_t - \tau(r)_x=0
\end{cases}, \quad
\begin{matrix}
 p(t,0)=0 & r(t,1) = \tau^{-1}(\taubar(t))
\\
p(0,x)=p_0(x)& r(0,x)= r_0(x)
\end{matrix}
\end{equation}
  The nonlinearity $\tau\in C^3(\mathbb R)$ is chosen to have the following properties.
\begin{itemize}
\item[($\tau$-i)] $ c_1 \le \tau'(r) \le  c_2$ for some $c_1, c_2 >0$ and all $r \in \mathbb R$;
\item[($\tau$-ii)] $\tau''(r)  \neq 0$ for all $r \in \mathbb R$;
\item[($\tau$-iii)] $\tau''(r), \tau'''(r) \in L^2(\mathbb R)\cap L^\infty(\mathbb R)$.
\end{itemize}
We also assume that
$\taubar : \mathbb R_+ \to \mathbb R$ is smooth. Moreover, there is a time $T_\star$   such that $\taubar'(t) = 0$ for all $t \ge T_\star$. The initial data $r_0, p_0 \in L^2(0,1)$ are compatible with the boundary conditions.
\begin{oss}
Conditions ($\tau$-i) and ($\tau$-ii) ensure that the system is strictly hyperbolic and genuinely nonlinear, respectively. Condition ($\tau$-iii) is used later on to ensure some boundedness properties of the Lax entropies.
\end{oss}

\begin{theorem}\label{th-main}
  System \eqref{eq:psystem} admits a weak solution $\bar u = (\bar r, \bar p)$
  in the sense of \eqref{eq:maineqr} and \eqref{eq:maineqp}, such that
$\bar u\in  L^\infty(0,T; L^2(0,1))$, 
   $\bar u(0,x) =u_0(x)$ for a.e. $x$;
 and it  satisfies the Clausius inequality:
\begin{equation} 
\mathcal F(\bar u(t))- \mathcal F(u_0) \le  W(t), \qquad  \forall t \in [0,T] 
\end{equation}
with $W(t)$ as in \eqref{eq:lavoro}. Furthermore $\bar u$
satisfies the local Lax entropy condition in the sense
specified in section \ref{sec:lax}.
\end{theorem}

\section{Viscous approximation and energy estimates}
\label{sec:visc-appr}
We consider the following parabolic approximation of the hyperbolic system \eqref{eq:psystem}
\begin{equation}
  \label{eq:vpsystem}
\begin{cases}
r^\delta_t- p^\delta_x= \deltauno  r^\delta_{xx} 
 \\
p^\delta_t-  \tau(r^\delta)_x = \deltadue  p^\delta_{xx}
\end{cases}, \qquad(t,x) \in \mathbb R_+ \times [0,1]
\end{equation}
for $\deltauno > 0$, with the boundary conditions:
\begin{equation}\label{eq:vbc}
 p^\delta(t,0)=0, \quad r^\delta(t,1) = \tau^{-1}(\taubar(t)),
\quad   p^\delta_x(t,1)= 0, \quad    r^\delta_x(t,0)= 0,
\end{equation}
and initial data:
\begin{equation}
p^\delta(0,x) =p_0^\delta(x), \quad r^\delta(0,x)=r_0^\delta(x).\label{eq:vic}
\end{equation}
The initial data $r^\delta_0, p^\delta_0 \in C^\infty([0,1])$ are mollifications of $r_0$ and $p_0$
compatible with the boundary conditions:
\begin{equation}
  p_0^\delta(0) = 0 \quad r_0^\delta(1) = \tau^{-1}(\taubar(0)), \quad \partial_x p_0^\delta(1) =0 \quad \partial_ x r_0^\delta(0) = 0.
  \label{eq:vic1}
\end{equation}
Moreover,  there is $C$ independent of $\delta$ such that
\begin{equation}\label{eq:vic2}
\|r_0^\delta \|_{L^2}+\|p_0^\delta \|_{L^2}+\|\sqrt \delta \partial_x r_0^\delta \|_{L^2}.
+ \| \sqrt \delta \partial_x p_0^\delta \|_{L^2} \le C
%
\end{equation}
and $(r_0^\delta, p_0^\delta) \to (r_0, p_0)$ strongly in  $L^2(0,1)$.

As shown in \cite{AlasioMarchesani} in a more general setting, this system admits a global classical solution $(r^\delta, p^\delta)$, with 
$$ r^\delta, p^\delta \in C^1(\mathbb R_+ ; C^0([0,1]) )\cap C^0(\mathbb R_+; C^2([0,1])).
$$
\begin{oss}
\begin{itemize}
\item[(i)] We added two extra Neumann conditions, namely $p_x^\delta(t,1)=r_x^\delta(t,0)=0$.
  These conditions reflect the conservative nature of the viscous perturbation, and
  are required in order to obtain the correct production of free energy. 
  
\item[(ii)] One could introduce a \emph{nonlinear} viscosity term: $\deltauno \tau(r^\delta)_{xx}$ in place of $\deltauno r^\delta_{xx}$.
  This is a term which comes naturally from a microscopic derivation of system \eqref{eq:vpsystem}, as described in the introduction (see also \cite{MarchesaniDiffusive}). Nevertheless, this does not  drastically change the problem, thus we shall consider only the linear viscosity
  $\deltauno  r^\delta_{xx}$.
\end{itemize}
\end{oss}




\begin{theorem}[Energy estimate] \thlabel{prop:energy}
There there is a constant $C>0$ independent of $t$
and $\delta$ such that
\begin{equation}
\int_0^1 |u^\delta(t,x)|^2 dx +
\delta \int_0^t \int_0^1 |u_x^\delta(s,x)|^2 dx ds  \le C
\end{equation}
for all $t \ge 0 $ and $\delta >0$.
\end{theorem}

\begin{proof}
Let $F$ be a primitive of $\tau$ such that $ \dfrac{c_1}{2} r^2 \le F(r) \le \dfrac{c_2}{2} r^2$. By a direct calculation we have
\begin{align} \label{eq:preclausius}
 \int_0^1 \left( \frac{(p^\delta)^2}{2}+ F(r^\delta) \right) dx  \biggr |_{t=0}^{t=T} +\int_0^T \int_0^1 \left( \deltauno (r^\delta_x)^2 + \deltadue (p_x^\delta)^2 \right) dx dt  = \int_0^T \taubar(t)  \int_0^1 r_t^\delta dx dt 
\\
=  \left( \taubar(t) \int_0^1 r^\delta dx \right) \biggr |_{t=0}^{t=T} - \int_0^T \taubar'(t) \int_0^1 r^\delta dx  dt.
\end{align}
Write, for some $\varepsilon > 0$ to be chosen later,
\begin{equation}
\taubar(T) \int_0^1 r^\delta(T,x)dx  \le |\taubar(T)|  \left( \frac{1}{2\varepsilon} + \frac{\varepsilon}{2}\int_0^1(r^\delta)^2(T,x) dx\right) \le \frac{C_\taubar}{2\varepsilon}+ \frac{C_\taubar \varepsilon}{2} \int_0^1 (r^\delta)^2(T,x)dx
\end{equation}
where $C_\taubar = \sup_{t \ge 0} \left( |\taubar(t)|+ |\taubar'(t)| \right)$ depends on $\taubar$ only.

Using $F(r) \ge \dfrac{c_1}{2} r^2$  we obtain
\begin{align}
 \left(\frac{c_1}{2}- \frac{C_\taubar \varepsilon}{2}\right) \int_0^1 (r^\delta)^2(T,x)dx + \frac{1}{2}\int_0^1 (p^\delta)^2(T,x)dx &+\int_0^T \int_0^1 \left( \deltauno (r^\delta_x)^2 + \deltadue (p_x^\delta)^2 \right) dx  dt \nonumber
\\
& \le \frac{C_\taubar}{2\varepsilon} + C_0 - \int_0^T \taubar'(t) \int_0^1 r^\delta(t,x) dx dt. 
\end{align} 
Recall that there is $T_\star >0$ such that $\taubar'(t) = 0$ for $t \ge T_\star$. Then, for $T < T_\star$, we write
\begin{align}
 \left(\frac{c_1}{2}- \frac{C_\taubar \varepsilon}{2}\right) \int_0^1 (r^\delta)^2(T,x)dx &+ \frac{1}{2}\int_0^1 (p^\delta)^2(T,x)dx +\int_0^T \int_0^1 \left( \deltauno (r^\delta_x)^2 + \deltadue (p_x^\delta)^2 \right) dx  dt  \nonumber
\\
& \le  \frac{C_\taubar}{2\varepsilon} + C_0 + \frac{C_\taubar^2}{2} T + \frac{1}{2} \int_0^T \int_0^1 (r^\delta)^2(t,x) dx 
\end{align}
\begin{align}
\le  \frac{C_\taubar}{2\varepsilon} + C_0 + \frac{C_\taubar^2}{2} T &+ \frac{1}{2} \int_0^T \int_0^1\left( (r^\delta)^2(t,x) +(p^\delta)^2(t,x) \right)dx dt 
\\
&+ \frac{1}{2}\int_0^T \int_0^t \int_0^1 \left( \deltauno (r^\delta_x)^2 + \deltadue (p_x^\delta)^2 \right) dxds dt \nonumber
\end{align}
where $C_0$ depends on the initial data only. Choosing $\varepsilon = c_1/(2C_\taubar)$ gives
\begin{equation}
\frac{c_1}{4} J(T) \le C_0 +\frac{C_\taubar^2}{c_1}+ \frac{C_\taubar^2}{2}T + \frac{1}{2} \int_0^T J(t) dt,
\end{equation}
where
\begin{equation}
J(t) =  \int_0^1 \left( (r^\delta)^2(t,x)+(p^\delta)^2(t,x)\right)dx + \int_0^t \int_0^1 \left(\deltauno (r^\delta_x)^2 + \deltadue (p_x^\delta)^2 \right) dx ds.
\end{equation}
We apply Gronwall's inequality. This, together with $T < T_\star$, gives
\begin{align}
J(T)  & \le  \frac{4c_1 C_0+ 2C_\taubar^2(2+  c_1 T)}{c_1^2} \exp \left(\frac{2T}{c_1} \right) \nonumber 
\\
&  \le   \frac{4c_1 C_0+ 2C_\taubar^2(2+  c_1 T_\star)}{c_1^2} \exp \left(\frac{2T_\star}{ c_1} \right) := C_0(c_1, \taubar),
\end{align}
for  all $T \in [0,T_\star)$, where $C_0(c_1, \taubar)$ is independent of $T$ and $\delta$

On the other hand, if $T \ge T_\star$, we have
\begin{align}
 \left(\frac{c_1}{2}- \frac{C_\taubar \varepsilon}{2}\right) \int_0^1 (r^\delta)^2(T,x)dx& + \frac{1}{2}\int_0^1 (p^\delta)^2(T,x)dx +\int_0^T \int_0^1 \left( \deltauno(r^\delta_x)^2 + \deltadue (p_x^\delta)^2 \right) dx  dt \nonumber
\\
& \le \frac{C_\taubar}{2\varepsilon} + C_0 - \int_0^{T_\star} \taubar'(t) \int_0^1 r^\delta(t,x) dx dt
\end{align}
and the integral at the right-hand side is uniformly bounded in $T$, $\deltauno$ and $\deltadue$, since
\begin{align}
- \int_0^{T_\star} \taubar'(t) \int_0^1 r^\delta(t,x) dx dt & \le C_\taubar \int_0^{T_\star} \int_0^1 |r^\delta(t,x)| dx dt \nonumber
\\
& \le C_\taubar T_\star \left( \frac{1}{T_\star} \int_0^{T_\star} \int_0^1 (r^\delta)^2(t,x)dxdt \right)^{1/2} \nonumber
\\
& \le C_\taubar \sqrt{T_\star} \left( \int_0^{T_\star} J(t) dt \right)^{1/2} \nonumber
\\
& \le C_\taubar T_\star \sqrt{C_0(c_1, \taubar)} 
\end{align}
\end{proof}

From \eqref{eq:preclausius} we also immediately obtain the following
\begin{cor}[Viscous Clausius inequality] \thlabel{prop:clausius}
\begin{equation}
\mathcal F(u^\delta(t)) - \mathcal F(u_0^\delta) \le - \int_0^t \taubar'(s) \int_0^1 r^\delta (s,x)dx + \taubar(t) \int_0^1 r^\delta(t,x) dx - \taubar(0)\int_0^1 r_0^\delta(x)dx.
\end{equation}
\end{cor}

\section{$L^2$ Young measures and compensated compactness}
\label{sec:young}

Throughout this section, for any fixed $T >0$ let $u^\delta(t,x) := (r^{\delta}(t,x), p^{\delta}(t,x))$ be a strong solution of \eqref{eq:vpsystem} on $Q_T$.
 By  \thref{prop:energy} and after a time integration over $[0,T]$ we obtain
 \begin{equation}\label{eq:l2b}
 \| u^\delta \|_{L^2(Q_T)} \le C
 \end{equation}
 for some $C$ independent of $\delta$. Thus  we can extract from $\{u^\delta\}_{\delta >0}$
 a subsequence that is weakly convergent in $L^2(Q_T)$.
Namely, up to a subsequence, there exists
$\bar u = (\bar r, \bar p) \in L^2(Q_T)$ such that
 \begin{equation} \label{eq:weakL2}
   \lim_{\delta \to 0} \int_{Q_T} u^\delta \varphi  = \int_{Q_T} \bar u \varphi ,
   \qquad \forall  \varphi \in  L^2(Q_T).
 \end{equation}
 All the limits $\delta\to 0$ taken below are intended along a chosen subsequence.

In this section we want to show that for any $\phi \in L^2(Q_T)$ we have
\begin{equation} \label{eq:nonlinearlimit}
\lim_{\delta \to 0}\int_{Q_T}  \phi(t,x) \tau(r^\delta(t,x) ) \; dx\; dt  =\int_{Q_T}   \phi(t,x)  \tau(\bar r(t,x)) \; dx\; dt.
\end{equation}
This is done using a $L^2$ version of the compensated compactness,
which is usually performed in $L^\infty$.

From the solution $u^\delta(t,x)$, we define the following Young measure on $Q_T \times \mathbb R^2$:
\begin{equation}
\nu_{t,x}^\delta:=  \delta_{u^\delta(t,x)},
\end{equation}
which is a Dirac mass centred at $u^\delta$, i.e.
\begin{equation*}
  \int_{Q_T} J(t,x) f(u^\delta(t,x)) \; dx\; dt  =
  \int_{Q_T} \int_{\mathbb R^2} J(t,x) f(\xi)d \nu^\delta_{t,x}(\xi)dxdt 
\end{equation*}
for all mesurable $J: Q_T \to \mathbb R$ and $f: \mathbb R^2 \to \mathbb R$.

Since we have $L^2$ bounds on $u^\delta$, we refer at $\nu^\delta_{t,x}$ as  a $L^2$-Young measure
\cite{ball1989version}.
In particular we have, from \eqref{eq:l2b}
\begin{equation}
  \label{eq:5}
  \int_{Q_T} \int_{\mathbb R^2} |\xi|^2 d \nu^\delta_{t,x}(\xi)\; dx\; dt \ \le C.
\end{equation}

We call $\mathcal Y$ the set of Young measures on $Q_T \times \mathbb R^2$
and we make it a metric space by endowing it with the Prohorov's metric.
By proposition 4.1 of \cite{StochasticYoung}, the set
\begin{equation}\label{eq:compact}
K_C := \left \{ \nu \in \mathcal Y : \int_{Q_T} \int_{\mathbb R^2} |\xi|^2 d \nu_{t,x}(\xi) dx dt \le C \right\}
\end{equation}
is compact in $\mathcal Y$.
Then, by the fundamental theorem for Young measures (\cite{ball1989version}, section 2),
there exists $\bar \nu_{t,x} \in \mathcal Y$
so that, up to a subsequence,
\begin{equation}
\label{eq:nuconv}
\lim_{\delta \to 0} \int_{Q_T} \int_{\mathbb R^2} J(t,x) f(\xi)d \nu^\delta_{t,x}(\xi)dxdt = \int_{Q_T}
\int_{\mathbb R^2} J(t,x) f(\xi)d \bar \nu_{t,x}(\xi)dxdt
\end{equation}
for all continuous and \emph{bounded} $J: Q_T \to \mathbb R$
and $f: \mathbb R^2 \to \mathbb R$.
We shall simply write $\nu^\delta \to \bar \nu$ in place of \eqref{eq:nuconv}.
By a simple adaptation of
proposition 4.2 of \cite{StochasticYoung}, \eqref{eq:nuconv} can be extended
to a function $f: \mathbb R^2\to \mathbb R$
such that $f(\xi)/|\xi|^2 \to 0$ as $|\xi| \to +\infty$.

In order 
to obtain \eqref{eq:nonlinearlimit},
  we need to prove that
  the limit Young measure $\bar \nu$ is a Dirac mass: $\bar \nu_{t,x} = \delta_{\bar u(t,x)}$,
  for some  $\bar u \in L^2(Q_T)$ and for almost  every $(t,x) \in Q_T$.
  This is done using the classical argument by Tartar and Murat.
\begin{defn}
A Lax entropy-entropy flux pair for system \eqref{eq:psystem} is a couple of differentiable functions $(\eta, q) : \mathbb R^2 \to \mathbb R^2$ such that
\begin{equation} \label{eq:laxsystem}
\begin{cases}
\eta_r + q_p =0
\\
\tau'(r) \eta_p + q_r= 0
\end{cases}.
\end{equation}
\end{defn}
We show that Tartar's equation holds for any two suitable entropy pairs $(\eta, q)$ and $(\eta', q')$ to be specified below and almost all $(t,x) \in Q_T$:
\begin{equation} \label{eq:tartar}
  \langle\eta q' -\eta' q, \bar \nu_{t,x} \rangle =
  \langle \eta, \bar \nu_{t,x} \rangle \langle q', \bar \nu_{t,x} \rangle-
  \langle \eta',  \bar \nu_{t,x} \rangle  \langle q,  \bar \nu_{t,x} \rangle,
\end{equation}
where
\begin{equation}
\langle f, \bar \nu_{t,x} \rangle := \int_{\mathbb R^2} f(\xi) d \bar\nu_{t,x}(\xi)
\end{equation}
for any measurable $f$. We employ the following argument due to Shearer \cite{Shearer1}.


Accordingly to  Lemma 2 in \cite{Shearer1}, there exists a family of half-plain supported entropy-entropy fluxes $(\eta,q)$
such that $\eta$ and $q$ are bounded together with their first and second derivatives. These are explicitly given as follows. We define $z(r) := \int_0^r \sqrt{\tau'(\rho)}d \rho$ and we define the Riemann coordinates $w_1 = p + z$, $w_2 = p-z$. We also pass from the dependent variables $\eta$, $q$ to $H$,$Q$ as follows:
\begin{align}
	\eta = \frac{1}{2}(\tau')^{-1/4}\left(H+Q\right)
	\\
	q = \frac{1}{2}(\tau')^{+1/4}\left(H-Q\right)
\end{align}
so that \eqref{eq:laxsystem} becomes
\begin{equation} \label{prova}
\begin{cases}
	H_{w_1}= a Q
	\\
	H_{w_2}  = - a Q
\end{cases},
\end{equation}
where
\begin{equation}
	a(w_1-w_2) = \frac{\tau''\left( r \left(\dfrac{w_1-w_2}{2}\right)\right)}{8 \left(\tau' \left(r\left(\dfrac{w_1-w_2}{2}\right)\right)\right)^{3/2}}.
\end{equation}
Then we fix $\bar w_1, \bar w_2 \in \mathbb R$ and we solve \eqref{prova} with Goursat data given on the lines $w_1 = \bar w_1$ and $w_2 = \bar w_2$:
\begin{align}
\begin{split}
H(\bar w_1, w_2)& = g(w_2)
\\
Q(w_1,\bar w_2) &= 0,
\end{split}
\end{align}
where $g$ is continuous and compactly supported. Then one can explicitly solve \eqref{prova} and get
\begin{align}
\begin{split}
H(w_1,w_2) &= g(w_2)+ \sum_{n=1}^\infty (\mathcal A^n g)(w_1,w_2)
\\
Q(w_1,w_2) &= - \int_{\bar w_2}^{w_2} a(w_1- v)H(w_1,v)dv,
\end{split}
\end{align}
where the operator $\mathcal A$ acts on functions $f \in L^1_{loc}(\mathbb R^2)$ as follows:
\begin{equation}
(\mathcal A f)(w_1,w_2) = -\int_{\bar w_1}^{w_1} \int_{\bar w_2}^{w_2}a(v-w_2)a(v-u)f(v,u)du dv.
\end{equation}
Finally, going back to $\eta$ and $q$ and using our assumptions on $\tau$ it is possible to show (\cite{Shearer1}, Lemma 2) that $\eta$ and $q$ are bounded, together with their first and second derivatives.

Now we have a suitable family of entropy-entropy flux pair, we use Tartar-Murat Lemma in order to derive Tartar's equation \eqref{eq:tartar}. We evaluate $(\eta, q)$ along the approximate solutions $u^\delta$ and compute the entropy production:
\begin{equation} \label{eq:laxProduction}
 \eta(u^\delta)_t+ q(u^\delta)_x= \delta \left(  \eta_r   r_x^\delta+   \eta_p  p_x^\delta \right)_x- \delta \left(  \eta_{rr} (r_x^\delta)^2+ \eta_{pp} (p_x^\delta)^2 +
 2 \eta_{rp}  r_x^\delta p_x^\delta\right)
\end{equation}
Since $\eta_r$ and $\eta_p$ are bounded and $\sqrt{\deltadue} r_x^\delta, \sqrt{\deltauno} p_x^\delta$
are bounded in $L^2(Q_T)$, we have 
\begin{equation}\label{eq:h-1}
\lim_{\delta \to 0}\delta \left(  \eta_r   r_x^\delta+   \eta_p  p_x^\delta \right)_x = 0 \qquad \text{in } H^{-1}(Q_T),
\end{equation}
 while
\begin{equation}
\left \| \deltauno \left( \eta_{rr} (r_x^\delta)^2+ \eta_{pp} (p_x^\delta)^2 +
2 \eta_{rp}  r_x^\delta p_x^\delta\right)  \right\|_{L^1(Q_T)} \le C
\end{equation}
uniformly with respect to 
 $\delta$.
 Thus we have have obtained  an equality of the form
 $$
 \eta(u^\delta)_t + q(u^\delta)_x =  \chi^\delta + \psi^\delta,
 $$
 where $\{\chi^\delta \}_{\delta >0}$ lies in a compact set of $H^{-1}(Q_T)$ and $\{\psi^\delta\}_{\delta >0}$ is bounded in $L^1(Q_T)$. Moreover, since $\eta$ and $q$ are bounded, $\{\eta(u^\delta)_t + q(u^\delta)_x\}_{\delta >0}$ is bounded in $W^{-1,p}(Q_T)$ for some $p >2$.
 
Therefore, we  can apply Tartar-Murat and the div-curl lemma (cf \cite{dafermos2010hyperbolic}, Theorem 16.2.1 and Lemma 16.2.2) and obtain Tartar's equation \eqref{eq:tartar}.

The final step is to use Tartar's equation to prove that the support of the limit Young measure $\bar \nu_{t,x}$ is a point. This is done in lemmas 4 to 7 of \cite{Shearer1} and leads to the following
\begin{prop}
There exists a $\bar u \in L^2(Q_T)$ such that $\bar \nu_{t,x} = \delta_{\bar u(t,x)}$ for almost all $(t,x) \in Q_T$. Moreover, $u^\delta \to \bar u$ strongly in $L^p(Q_T)$ for any $p \in [1,2)$.
\end{prop}

\subsection{Regularity}
\begin{prop} \label{lem:infty}
  For the function $\bar u$ obtained in section \ref{sec:young}, 
  $$
  \bar u \in L^\infty(0,T; L^2(0,1)).
  $$
\end{prop}
\begin{proof}
Since $u^\delta \to \bar u$ in $L^p$ strong for $p<2$, we can extract a subsequence $\{u^{\delta_k}\}_{k \in \mathbb N}$ that converges pointwise to $\bar u$ for almost all $t$ and $x$. In particular, for almost all $t$, the sequence $u^{\delta_k}(t,x)$ converges for almost all $x$. Therefore, by Fatou lemma and Theorem \ref{prop:energy}, 
\begin{equation}
\int_0^1 |\bar u(t,x)|^2 dx \le \liminf_{k \to \infty} \int_0^1 |u^{\delta_k}(t,x)|^2 dx \le C
\end{equation}
for almost all $t \in [0,T]$.
\end{proof}
The proof is of next lemma is standard and therefore omitted.
\begin{lem}
Let $a(t) := \tau^{-1}(\taubar(t))$. Then, the solutions $(r^\delta, p^\delta)$ of the viscous system \eqref{eq:vpsystem} can be written as follows:
\begin{align}
r^\delta(t,x)  =a(t)&+ \int_0^1 G^\delta_r(x,x',t)( r_0^\delta(x')-a(0)) d x' +
\\ 
&+ \int_0^t \int_0^1 G^\delta_r(x,x',t-t') (\partial_{x'}p^\delta(t',x')-a(t'))dx' dt' \nonumber
\end{align}
\begin{align}
p^\delta(t,x)  = \int_0^1 G^\delta_p(x,x',t) p_0^\delta(x') d x' + \int_0^t \int_0^1 G^\delta_p(x,x',t-t')\partial_{x'}\tau(r^\delta(t',x')) dx' dt'
\end{align}
where the $G^\delta_r$ and $G^\delta_p$ are Green functions of the heat  operator $\partial_t - \delta \partial_{xx}$ with homogeneous boundary conditions:
\begin{align}
G^\delta_r(1,x',t) = \partial_x G^\delta_r(0,x',t) = 0
\\
G^\delta_p(0,x',t) = \partial_x G^\delta_p(1,x',t) =0
\end{align}
for all $x' \in [0,1]$, $t \ge 0$ and $\delta > 0$.
\end{lem}
The Green's functions $G^\delta_r(x,x',t)$ and $G_p^\delta(x,x',t)$ are symmetric under the exchange of $x$ and $x'$. Moreover we have the following identities
\begin{align}
\partial_x G^\delta_p(x,x',t) = - \partial_{x'} G_r^\delta(x,x',t),
\\
\partial_x G^\delta_r(x,x',t) = - \partial_{x'} G_p^\delta(x,x',t).
\end{align}
\begin{oss}
The functions $G^\delta_r$ and $G_p^\delta$ have the following explicit forms:
\begin{align}
G_p^\delta(x,x',t) &= { \frac 12} \sum_{n\text{ odd}} e^{-t\delta \lambda_n} \sin \left(\sqrt{\lambda_n}x \right)\sin\left(  \sqrt{\lambda_n}x' \right)
\\
G_r^\delta(x,x',t) &= { \frac 12} \sum_{n\text{ odd}} e^{-t\delta \lambda_n} \cos \left(\sqrt{\lambda_n}x \right)\cos \left(  \sqrt{\lambda_n} x' \right),
\end{align}
with $\lambda_n =\dfrac{n^2 \pi^2}{4}$.
\end{oss}
\begin{prop}\label{cont}
For any $\phi \in C^1([0,1])$, the application
\begin{equation}
t \mapsto I_\phi(t) := \int_0^1 \phi(x) \bar u(t,x)dx
\end{equation}
is Lipschitz continuous.
Consequently $\bar u(t,\cdot) \in L^2(0,1)$ for all $t\ge 0$.
\end{prop}

\begin{proof}
We prove the statement for $\bar p$, as the proof for $\bar r$ is similar. Furthermore, we prove the proposition only between $0$ and $t$, as in the general case, say between $t_1$ and $t$, it is enough to replace the initial term $p_0^\delta(x)$ with $p^\delta(t_1,x)$. We let
\begin{equation}
I_\phi^\delta (t) :=  \int_0^1 \phi(x) p^\delta(t,x) dx
\end{equation}
and evaluate
\begin{align}
I_\phi(t)-I_\phi(0)= & \int_0^1 \int_0^1 \phi(x) G^\delta_p(x,x',t) p_0^\delta(x') d x' dx- \int_0^1 \phi(x) p_0^\delta(x) d x +
\\
&+ \int_0^t \int_0^1 \int_0^1 \phi(x) G_p^\delta(x,x',t-t') \partial_{x'}\tau(r^\delta(t',x'))dx dx' dt' \nonumber
\end{align}
\begin{align} \label{eq:primo}
=& {\ \int_0^1 \int_0^1 \phi(x) \left[ G_p^\delta(x,x',t) - \delta(x,x')\right] p_0^\delta(x')\ dx' dx }
\\
  &+ \int_0^t \int_0^1 \int_0^1 \phi(x) \partial_x G^\delta_r(x,x',t-t') \tau(r^\delta(t',x')) dx dx' dt'
    + \int_0^t \int_0^1 \phi(x) G_p^\delta(x,1,t-t') \taubar(t') dx dt', \nonumber
\end{align}
where we have used the symmetry of $G^\delta_p$
as well as the property $\partial_{x'}G_p^\delta = - \partial_x G_r^\delta$.
The boundary term is estimated as
\begin{equation}
  \left |\int_0^t \int_0^1 \phi(x) G_p^\delta(x,1,t-t') \taubar(t') dx dt' \right| \le
  \|\taubar\|_\infty \int_0^t \left| \int_0^1 \phi(x) G_p^\delta(x,1,t-t') dx \right| dt'
  \le t  \|\taubar\|_\infty  \| \phi \|_{L^2}.
\end{equation}

 We estimate the term involving $\tau$ as
\begin{equation}
  \label{eq:7}
  \begin{split}
    \left | \int_0^t \int_0^1 \int_0^1 \phi(x) \partial_x G^\delta_r(x,x',t-t') \tau(r^\delta(t',x')) dx dx' dt' \right |
   \\ = 
 \left| \frac 12 \sum_{n \text{odd}} 
  \int_0^1 \sqrt{\lambda_n} \sin(\sqrt{\lambda_n}x) \phi(x) dx  \int_0^t dt' e^{-(t-t')\lambda_n \delta}
  \int_0^1 \cos(\sqrt{\lambda_n}x') \tau(r^\delta(t',x')) dx' \right|\\
  \le  \| \phi' \|_{L^2} \left[ \frac 12 \sum_{n \text{odd}} \left(\int_0^t dt' e^{-(t-t')\lambda_n \delta}
      \int_0^1 \cos(\sqrt{\lambda_n}x') \tau(r^\delta(t',x')) dx'\right)^2 \right]^{1/2}\\
\le  \| \phi' \|_{L^2} \left[ \frac 12 \sum_{n \text{odd}}
    \left(\frac{1-e^{-2 t\lambda_n \delta}}{2\lambda_n\delta} \right) \int_0^t dt'
    \left(\int_0^1 \cos(\sqrt{\lambda_n}x') \tau(r^\delta(t',x')) dx'\right)^2\right]^{1/2}\\
\le \| \phi' \|_{L^2} \left[ \frac{t}2 \int_0^t dt' \sum_{n \text{odd}}
    \left(\int_0^1 \cos(\sqrt{\lambda_n}x') \tau(r^\delta(t',x')) dx'\right)^2\right]^{1/2}\\
 = \| \phi' \|_{L^2} \left[ t \int_0^t dt' \int_0^1 \tau(r^\delta(t',x'))^2 dx' \right]^{1/2}
\end{split}
\end{equation}
\begin{align}
  & \le t \|\phi'\|_{L^2} \left\| \int_0^1 \tau(r^\delta(\cdot,x'))^2dx'  \right\|_{L^\infty(0,T)}^{1/2}
    \le C  t  \|\phi'\|_{L^2}
\end{align}
where $C$ is independent of $t$ and $\delta$.

In order to estimate the first term of \eqref{eq:primo} we write
{
  \begin{align}
    \label{eq:9}
      \int_0^1 \int_0^1& \phi(x) \left[ G_p^\delta(x,x',t) - \delta(x,x')\right] p_0^\delta(x')\ dx' dx\\
   &=    \frac 12 \sum_{n \text{odd}} 
  \int_0^1 \sin(\sqrt{\lambda_n}x) \phi(x) dx  \left(e^{-t\lambda_n \delta} - 1\right) \nonumber
  \int_0^1 \sin(\sqrt{\lambda_n}x') p_0^\delta(x') \; dx'\\
  &\le  t\delta  \frac 12 \sum_{n \text{odd}} \lambda_n \nonumber
  \left|\int_0^1 \sin(\sqrt{\lambda_n}x) \phi(x) dx \right|
  \left|\int_0^1 \sin(\sqrt{\lambda_n}x') p_0^\delta(x') \; dx'\right|\\
  &\le t \|\phi'\|_{L^2} \|\delta \partial_x p_0^\delta\|_{L^2} \le  t \|\phi'\|_{L^2} C \nonumber
   \end{align}
     }
  where we have used the assumption that $\{\sqrt \delta \partial_x p_0^\delta\}_{\delta >0}$ is bounded in $L^2(0,1)$.
  We have also assumed, without loss of generality, $\delta \le 1$.

  Putting everything together, we have obtained
\begin{equation}
\left| I^\delta_\phi(t)-I^\delta_\phi(0) \right | \le t C \left(\|\phi'\|_{L^2} + \|\phi\|_{L^2}\right)
\end{equation}
for some constant $C$ independent of $t$ and $\delta$.
This leads to the conclusion after passing to the limit $\delta \to 0$.
\end{proof}

\section{Proof of theorem \ref{th-main} and Clausius inequality}
\label{sec:proof-theorem-refth}

 All is left to prove is that the function  $\bar u$ obtained in the previous section is to a weak solution
 of the hyperbolic system \eqref{eq:psystem},
 in the sense of Section \ref{sec:definition}.
 Let $\psi \in C^1(Q_T)$  with $\psi(t,0)=0$ for all $t \in [0,T]$.
 Then, for any $t \in[0, T]$ we have
\begin{align} \label{eq:weak_parabolic}
0 & =\int_0^t\int_0^1\left(  \psi  p_s^\delta- \psi \tau(r^\delta)_x - \deltadue  \psi  p^\delta_{xx} \right)dxds \nonumber
\\
&=\int_0^1\psi(t,x)p^\delta(t,x)dx-\int_0^1\psi(0,x)p_0^\delta(x)dx+
\\
& - \int_0^t \int_0^1\left(\psi_s p^\delta- \psi_x \tau(r^\delta)-
  \deltadue  \psi_x  p_x^\delta \right)dx ds-\int_0^\infty \psi(t,1) \taubar(t) dt. \nonumber
\end{align} 
where we have used the initial-boundary conditions $\tau(r^\delta(t,1)) = \bar \tau(t)$
and $  p^\delta_x(t,1) =0$, $p^\delta(0,x)=p^\delta_0(x)$ as well as $\psi(t,0)=0$.
Since $p_0^\delta$ converges to $p_0$ in $L^2(0,1)$, we have
\begin{equation}
\lim_{\delta \to 0} \int_0^1 \psi(0,x) p_0^\delta(x)dx = \int_0^1 \psi(0,x) p_0(x) dx.
\end{equation}
Furthermore, \thref{prop:energy} implies  $\sqrt{\deltadue} p_x^\delta \in L^2(Q_T)$,
consequently 
$\int_0^t \int_0^1 \deltadue \psi_x p_x^\delta dx ds$
vanishes as $\deltadue \to 0$.
Moreover, \eqref{eq:weakL2} implies, along the subsequence that defines $\bar u = (\bar r, \bar p)$,
\begin{equation}
\lim_{\delta \to 0}\int_0^t \int_0^1  \psi_t p^\delta dxds  = \int_0^t \int_0^1  \psi_s \bar p dxdt,
\end{equation}
while by \eqref{eq:nonlinearlimit} we have that
\begin{equation*}
\lim_{\delta \to 0}  \int_0^t \int_0^1 \psi_x \tau(r^\delta) dx ds = \int_0^t \int_0^1 \psi_x \tau(\bar r) dx ds, 
\end{equation*}
so that \eqref{eq:maineqp} is satisfied. The \eqref{eq:maineqr} is linear and it follows similarly.

\begin{prop}
The solution $\bar u$ satisfies Clausius inequality
\begin{equation}
  \label{eq:v20}
   \mathcal F(\bar u(t)) - \mathcal F(u_0) \le  W(t)
\end{equation}
for 
all $t \in[ 0,T]$, where
\begin{equation}
W(t) =- \int_0^t \taubar'(s) \int_0^1 \bar r (s,x)dx + \taubar(t) \int_0^1 \bar r(t,x) dx - \taubar(0)\int_0^1 r_0(x)dx.
\end{equation}
\end{prop}
\begin{proof}
  By Proposition \ref{lem:infty}, Corollary \ref{prop:energy},  and Lemma \ref{cont},
  we have, for 
  all $t \in [0,T]$,
\begin{equation}
\int_0^1 \left( \frac{\bar p^2(t,x)}{2} + F(\bar r(t,x))\right)dx  \le \liminf_{k \to \infty} \int_0^1  \left( \frac{ ( p^{\delta_k})^2(t,x)}{2} + F( r^{\delta_k}(t,x))\right)dx
\end{equation}
\begin{align}
 &\le  \lim_{k \to \infty} \left( \mathcal F(u_0^{\delta_k})- \int_0^t \taubar'(s) \int_0^1 r^{\delta_k} (s,x)dxds + \taubar(t) \int_0^1 r^{\delta_k}(t,x) dx - \taubar(0)\int_0^1 r_0^{\delta_k}(x)dx \right)
\\
& = \mathcal F(u_0)- \int_0^t \taubar'(s) \int_0^1 \bar r (s,x)dx ds+ \taubar(t) \int_0^1 \bar r(t,x) dx - \taubar(0)\int_0^1 r_0(x)dx,
\end{align}
where we have used the fact that $u_0^\delta$ converges to $u_0$ in $L^2$ strongly
in order to conclude that $\mathcal F(u^\delta_0) \to \mathcal F(u_0)$.
Moreover, all the integrals are well defined, since the application
$$
t \mapsto \int_0^1 \bar r(t,x) dx
$$
is continuous.
\end{proof}
Thanks to the Clausius inequality,
the solutions we have constructed are natural candidates for being
the thermodynamic entropy solution of the equation 
\eqref{eq:vpsystem} and one can conjecture that such limit is unique. 


\section{Lax entropy condition}
\label{sec:lax}

From \eqref{eq:laxProduction}, if $\eta$ is a convex Lax entropy, we have that
\begin{equation} \label{eq:laxProduction2}
 \eta(u^\delta)_t+ q(u^\delta)_x\ge \delta \left(  \eta_r   r_x^\delta+   \eta_p  p_x^\delta \right)_x.
\end{equation}
If $\eta$ grows at most quadratically, the right hand side of \eqref{eq:laxProduction2} vanished
in $H^{-1}(Q_T)$, and by \eqref{eq:h-1}, for the limit we have that $\eta(u)_t+ q(u)_x \ge 0$
as a distribution in $H^{-1}(Q_T)$.
This is the usual \emph{local} characterization of weak entropy solutions, that is independent of the
boundary conditions and does not give informations of the \emph{behaviour} at the boundary.
Our solutions obtained from viscosity approximation satisfy such local entropy condition.

Our point is that this local characterization should be
implemented by the global entropy production for the
entropy given by the free energy $\eta(u) = \mathcal F(u)$,
i.e. the Clausius inequality.

\section*{Acknowledgments} 
This work has been partially supported by the grants ANR-15-CE40-0020-01 LSD 
of the French National Research Agency.
We thank Olivier Glass for very helpful discussions.
We also thank an anonomous referee, whose critics and comments helped us to improve the first version of this article.

\addcontentsline{toc}{chapter}{References}

\renewcommand{\bibname}{References}
\nocite{*}
\bibliography{bibliografia}

\begin{thebibliography}{10}

\bibitem{acquistapace1988quasilinear}
P~Acquistapace and B~Terreni.
\newblock On quasilinear parabolic systems.
\newblock {\em Mathematische Annalen}, 282(2):315--335, 1988.

\bibitem{AlasioMarchesani}
L~Alasio and S~Marchesani.
\newblock Global existence for a class of viscous systems of conservation laws.
\newblock {\em Nonlinear Differ. Equ. Appl.}, 26, 2019.

\bibitem{ball1989version}
J~M Ball.
\newblock A version of the fundamental theorem for young measures.
\newblock In {\em PDEs and continuum models of phase transitions}, pages
  207--215. Springer, 1989.

\bibitem{ball2013entropy}
John~M Ball and Gui-Qiang~G Chen.
\newblock Entropy and convexity for nonlinear partial differential equations,
  2013.

\bibitem{StochasticYoung}
F~Berthelin and J~Vovelle.
\newblock Stochastic isentropic {E}uler equations.
\newblock {\em to appear in Annales de l'ENS}, 2017.

\bibitem{BianchiniB}
S~Bianchini and A~Bressan.
\newblock Vanishing viscosity solutions of nonlinear hyperbolic systems.
\newblock {\em Annals of Mathematics}, 161:223--342, 2005.

\bibitem{chen-frid1999}
G-Q Chen and H~Frid.
\newblock Vanishing viscosity limit for initial-boundary value problems for
  conservation laws.
\newblock {\em Contemporary Mathematics}, 238:35--51, 1999.

\bibitem{chen2010vanishing}
G-Q~G Chen and M~Perepelitsa.
\newblock Vanishing viscosity limit of the navier-stokes equations to the euler
  equations for compressible fluid flow.
\newblock {\em Communications on Pure and Applied Mathematics},
  63(11):1469--1504, 2010.

\bibitem{dafermos2010hyperbolic}
C~M Dafermos.
\newblock Hyperbolic conservation laws in continuum physics, volume 325 of
  grundlehren der mathematischen wissenschaften [fundamental principles of
  mathematical sciences], 2010.

\bibitem{diperna}
R~Di~Perna.
\newblock Convergence of approximate solutions to conservation laws.
\newblock {\em Archiv Rational Mech. Anal.}, 82:27--70, 1983.

\bibitem{evans2004survey}
Lawrence Evans.
\newblock A survey of entropy methods for partial differential equations.
\newblock {\em Bulletin of the American Mathematical Society}, 41(4):409--438,
  2004.

\bibitem{Fritz1}
J~Fritz.
\newblock Microscopic theory of isothermal elastodynamics.
\newblock {\em Archiv Rational Mech. Anal.}, 201(1):209--249, 2011.

\bibitem{lin92}
P~Lin.
\newblock Young measures and an application of compensated compactness to
  one-dimensional nonlinear elastodynamics.
\newblock {\em Transactions of the AMS}, 329(1):377--413, January 1992.

\bibitem{MarchesaniDiffusive}
S~Marchesani.
\newblock Hydrodynamic limit for a diffusive system with boundary conditions.
\newblock {\em Preprint, https://arxiv.org/abs/1903.08576}, 2019.

\bibitem{MO2}
S~Marchesani and S~Olla.
\newblock In preparation.

\bibitem{MO1}
S~Marchesani and S~Olla.
\newblock Hydrodynamic limit for an anharmonic chain under boundary tension.
\newblock {\em Nonlinearity}, 31(11):4979–5035, 2018.

\bibitem{murat1978compacite}
F~Murat.
\newblock Compacit{\'e} par compensation.
\newblock {\em Annali della Scuola Normale Superiore di Pisa-Classe di
  Scienze}, 5(3):489--507, 1978.

\bibitem{murat1981compacite}
F~Murat.
\newblock Compacit{\'e} par compensation: condition n{\'e}cessaire et
  suffisante de continuit{\'e} faible sous une hypothese de rang constant.
\newblock {\em Annali della Scuola Normale Superiore di Pisa-Classe di
  Scienze}, 8(1):69--102, 1981.

\bibitem{olla2014microscopic}
S~Olla.
\newblock Microscopic derivation of an isothermal thermodynamic transformation.
\newblock In {\em From Particle Systems to Partial Differential Equations},
  pages 225--238. Springer, 2014.

\bibitem{otto}
F~Otto.
\newblock Initial-boundary value problem for a scalar conservation law.
\newblock {\em Comptes rendus de l'Acad{\'e}mie des sciences. S{\'e}rie 1,
  Math{\'e}matique}, 322:729--734, 1996.

\bibitem{serre86}
D~Serre.
\newblock La compacité par compensation pour les systèmes hyperboliques non
  linéaires de deux équations à une dimension d'espace.
\newblock {\em J. Math. Pures Appl.}, 65(4):423--468, 1986.

\bibitem{SerreShearer}
D~Serre and J~Shearer.
\newblock Convercenge with physical viscosity for nonlinear elasticity.
\newblock {\em Preprint, unpublished}, 1994.

\bibitem{Shearer1}
J~W Shearer.
\newblock Global existence and compactness in ${L}^p$ for the quasi-linear wave
  equation.
\newblock {\em Communications in Partial Differential Equations},
  19(11):1829--1878, 1994.

\bibitem{tartar1979compensated}
L~Tartar.
\newblock Compensated compactness and applications to partial differential
  equations.
\newblock In {\em Nonlinear Analysis and Mechanics: Heriot-Watt symposium},
  volume~4, pages 136--212, 1979.

\bibitem{tartar1983compensated}
L~Tartar.
\newblock The compensated compactness method applied to systems of conservation
  laws.
\newblock In {\em Systems of Nonlinear Partial Differential Equations}, pages
  263--285. Springer, 1983.

\end{thebibliography}
\bibliographystyle{plain}

\noindent
{Stefano Olla\\
CEREMADE, UMR-CNRS, Universit\'e de Paris Dauphine, PSL Research University}\\
{\footnotesize Place du Mar\'echal De Lattre De Tassigny, 75016 Paris, France}\\
{\footnotesize \tt olla@ceremade.dauphine.fr}\\
\\
{Stefano Marchesani\\
GSSI, \\
{\footnotesize Viale F. Crispi 7, 67100 L'Aquila, Italy}}\\
{\footnotesize \tt stefano.marchesani@gssi.it}\\

\end{document}